\documentclass[11pt,a4paper,leqno]{amsart}

\usepackage[latin1]{inputenc}
\usepackage[T1]{fontenc}
\usepackage{amsfonts}
\usepackage{amsmath}
\usepackage{amssymb}
\usepackage{eurosym}
\usepackage{mathrsfs}
\usepackage{palatino}
\usepackage{color}
\usepackage{esint}
\usepackage{url}

 \usepackage{amsrefs}
 
\usepackage{hyperref} 
\hypersetup{
    colorlinks=false,       
    linkcolor=blue,          
    citecolor=magenta,        
    filecolor=magenta,      
    urlcolor=cyan           
}

\newcommand{\R}{\mathbb{R}}

\newcommand{\diam}{\operatorname{diam}}

\numberwithin{equation}{section}

\newcommand{\Car}[0]{\operatorname{Car}}

\swapnumbers
\theoremstyle{plain}
\newtheorem{thm}[equation]{Theorem}

\theoremstyle{definition}
\newtheorem{defn}[equation]{Definition}

\theoremstyle{remark}
\newtheorem{rem}[equation]{Remark}

\author{Henri Martikainen}
\address[H.M.]{Department of Mathematics and Statistics, University of Helsinki, P.O.B. 68, FI-00014 Helsinki, Finland}
\email{henri.martikainen@helsinki.fi}
\thanks{Research of H.M. is supported by the Academy of Finland through the grant
Multiparameter dyadic harmonic analysis and probabilistic methods. }

\author{Mihalis Mourgoglou}
\address[M.M.]{Departament de Matem\`atiques, Universitat Aut\`onoma de Barcelona and Centre de Reserca Matem\` atica, Edifici C Facultat de Ci\`encies, 08193 Bellaterra (Barcelona)}
\email{mmourgoglou@crm.cat}
\thanks{Research of M.M. is supported by the ERC grant 320501 of the European Research Council (FP7/2007-2013).}

\makeatletter
\@namedef{subjclassname@2010}{%
  \textup{2010} Mathematics Subject Classification}
\makeatother


\title{Note about square function estimates and uniformly rectifiable measures}

\begin{document}

\begin{abstract}
We generalise and offer a different proof of a recent $L^2$ square function estimate on UR sets by
Hofmann, Mitrea, Mitrea and Morris. The proof is a short argument using the $\alpha$-numbers of Tolsa.
\end{abstract}

\maketitle

\section{Introduction}

We will deal with certain square function estimates involving the following class of kernels:
\begin{defn}
Let $\gamma_1, \gamma_2 > 0$.
We say that $S \in K_{\gamma_1, \gamma_2}(\R^d)$ if $S\colon \R^{d} \times \R^{d} \setminus \{(x,x)\colon\, x \in \R^{d}\} \to \R$ satisfies
for some $C < \infty$ that
\begin{displaymath}
|S(x,y)| \le \frac{C}{|x-y|^{\gamma_1}}
\end{displaymath}
and
\begin{displaymath}
|S(x,y) - S(x,y')| \le C\frac{|y-y'|^{\gamma_2}}{|x-y|^{\gamma_1+\gamma_2}}
\end{displaymath}
whenever $|y-y'| \le |x-y|/2$. For $\gamma > 0$ we set $K_{\gamma}(\R^d) = K_{\gamma, 1}(\R^d)$.
\end{defn}
Let $0 < n < d$ and $\mu$ be an $n$-ADR measure in $\R^d$.
We denote the support of the measure $\mu$ by $E$, i.e. $E = \textup{spt}\,\mu$. For convenience only we assume that
$d(E) = \infty$. The fact that $\mu$ is $n$-Ahlfors-David-regular (denote it by $n$-ADR) means that $\mu(B(x,r)) \sim r^n$ for every $x \in E$. Later on we will be concerned with the uniformly rectifiable measures:
\begin{defn}\label{defn:UR}
A measure $\mu$ in $\R^d$ is $n$-UR (uniformly rectifiable) if it is $n$-ADR and it satisfies the big pieces of Lipschitz images (BPLI) property.
This means that there should exist $\theta, M > 0$ such that for all $x \in E = \textup{spt}\,\mu$ and $r > 0$
we have a Lipschitz mapping $g$ from the ball $B_n(0,r) \subset \R^n$ to $\R^d$ with Lip$(g) \le M$ and
\begin{displaymath}
\mu(B(x,r) \cap g(B_n(0,r))) \ge \theta r^n.
\end{displaymath}
\end{defn}

Suppose $S \in K_{n+\beta, \gamma}(\R^d)$ for some $\beta,\gamma > 0$. 
For $f \in L^2(\mu) \cup L^{\infty}(\mu)$ and $x \in \R^{d} \setminus E$ we define
\begin{displaymath}
T_{S, \mu} f(x) = \int_E S(x,y)f(y)\,d\mu(y).
\end{displaymath}
In this ADR setting the $T1$ theorem, Theorem 3.2 of \cite{HMMM}, says that the square function estimate
\begin{equation}\label{eq:SF}
\int_{\R^{d} \setminus E} |T_{S,\mu}f(x)|^2 d(x, E)^{2\beta-(d-n)}\, dx \lesssim \int_E |f(y)|^2 \,d\mu(y), \qquad f \in L^2(\mu),
\end{equation}
is equivalent to
\begin{equation}\label{eq:t1}
\sup_{R \in \mathcal{D}(E)} \frac{1}{\mu(R)} \int_{\widehat R}  |T_{S,\mu}1(x)|^2 d(x, E)^{2\beta-(d-n)}\, dx < \infty.
\end{equation}
This condition involves some dyadic notation which we need to eventually explain carefully. However, before doing that we give a brief account on what we are aiming towards.

The $T1$ theorem is extremely useful for verifying the square function estimate \eqref{eq:SF}. However, the object
\begin{displaymath}
T_{S, \mu}1(x) = \int_E S(x,y)\,d\mu(y)
\end{displaymath} 
might not be so easy to get a hold of if $E$ is not something "geometrically simple". In the case that $\mu$ is not only $n$-ADR but also $n$-UR and $S \in K_{n+\beta}(\R^d)\, (= K_{n+\beta, 1}(\R^d))$, a Carleson type condition where one only needs to integrate over $n$-planes turns out to be sufficient for  \eqref{eq:SF}. Such a condition involving only the objects
\begin{displaymath}
T_{S, L}1(x) := \int_L S(x,y)\,d\mathcal{H}^n(y)
\end{displaymath}
can be preferable when $S$ is seen to have some special cancellation on $n$-planes $L$. Here $\mathcal{H}^n$ denotes the $n$-dimensional Hausdorff measure in $\R^d$.
The fact that a certain Carleson type condition involving only $T_{S,L}$ implies \eqref{eq:SF}
is the content of Theorem \ref{thm:main}. The theorem contains some of the results of David and Semmes \cite{DS} and Hofmann, Mitrea, Mitrea and Morris \cite{HMMM} (a much more detailed discussion is given after
stating the theorem). We emphasise the short proof inspired by the recent techniques of Tolsa and co-authors \cite{T1} and \cite{CGLT}.

Let us now introduce some relevant objects and definitions. First, we explain the dyadic notation related to the T1 condition \eqref{eq:t1}.
Let $\mathcal{D}(E)$ be a dyadic structure in $E$ (that is, a collection of David or Christ cubes).
This means that $\mathcal{D}(E) = \bigcup_j \mathcal{D}_j(E)$, and each cube (this is just terminology) $Q \in \mathcal{D}_j(E)$
satisfies $Q \subset E$, $c^{-1}2^{-j} \le \diam(Q) \le 2^{-j}$ and $\mu(Q) \sim 2^{-jn}$. We set $\ell(Q) = 2^{-j}$. These sets enjoy the usual structural properties of dyadic cubes
i.e. for two cubes $Q, R \in \mathcal{D}(E)$ either $Q \cap R = \emptyset$ or one of them is contained in the other. For a dyadic cube $R \in \mathcal D(E) $, $R^{(k)}$ denotes the unique cube $S \in \mathcal D(E)$ such that $R \subset S$ and $\ell(S) = 2^k\ell(R)$.

For a true cube $W \subset \R^{d}$ we denote its side length also by $\ell(W)$.
Let $\mathcal{W}$ denote the collection of maximal cubes $W$ from the standard dyadic grid of $\R^{d}$ for which there holds that $3W \subset \R^{d} \setminus E$.
Then we have that $d(x, E) \sim \ell(W)$ for every $x \in 2W$. To each $W \in \mathcal{W}$ we associate precisely one $Q(W) \in \mathcal{D}(E)$ for which $d(Q(W), W) \sim \ell(W) \sim \ell(Q(W))$.
For every $Q \in \mathcal{D}(E)$ we then define the Whitney region associated to $Q$ by setting
\begin{displaymath}
W_Q = \bigcup \{W \in \mathcal{W}\colon\, Q(W) = Q\}.
\end{displaymath}
The Carleson box $\widehat R$ is defined by
\begin{displaymath}
\widehat R = \mathop{\bigcup_{Q \in \mathcal{D}(E)}}_{Q \subset R} W_Q.
\end{displaymath}

The above is a way to produce Whitney regions which works in this generality. The exact way of producing them is not of great importance. Rather, it is the properties that they enjoy which we shall now list.
We have that the sets $W_Q$, $Q \in \mathcal{D}(E)$, are disjoint, $\R^{d} \setminus E = \bigcup_{W \in \mathcal{W}} W =  \bigcup_{Q \in \mathcal{D}(E)} W_Q$, $d(x,E) \sim \ell(Q)$
if $x \in W_Q$, and $|W_Q| \lesssim \ell(Q)^{d}$ (if $W_Q \neq \emptyset$ then $|W_Q| \sim \ell(Q)^{d}$). For later purposes we now also fix $M \sim 1$ so that $2W \subset B(c_Q, M\ell(Q))$
if $Q(W) = Q$. Here $c_Q$ denotes the centre of $Q$ -- it is a point in $Q$ such that $d(c_Q, E \setminus Q) \gtrsim \ell(Q)$. To each cube $Q \in \mathcal{D}(E)$ we also associate the ball
$B_Q = B(c_Q, 2M\ell(Q))$. 

The big pieces of Lipschitz images property stated in Definition \ref{defn:UR} seems to be the preferred definition of uniform rectifiability. However, it is equivalent to a huge plethora of different conditions.
For us the crucial one is the one using the so called $\alpha$-numbers of Tolsa \cite{T1}. Let us introduce this now. For two Borel measures $\sigma$ and $\nu$ in $\R^{d}$ and a closed ball $B \subset \R^{d}$ we set
\begin{displaymath}
d_B(\sigma, \nu) = \sup\Big\{ \Big| \int f \,d\sigma - \int f\,d\nu\Big|\colon \, \textup{Lip}(f) \le 1,\, \textup{spt}\,f\subset B\Big\}.
\end{displaymath}
For $Q \in \mathcal{D}(E)$, we define
\begin{displaymath}
\alpha(Q) = \frac{1}{\ell(Q)^{n+1}} \inf_{c \ge 0, L} d_{B_Q}(\mu, c\mathcal{H}^n_{\mid L}),
\end{displaymath}
where (recall that) $B_Q = B(c_Q, 2M\ell(Q))$, and the infimum is taken over all the constants $c \ge 0$ and all the $n$-planes $L$ for which $L \cap \frac{1}{2}B_Q \ne \emptyset$.
The constant $\alpha(Q)$ measures in a scale invariant way how close $\mu$ is to a flat $n$-dimensional measure in the ball $B_Q$.
The key result of \cite{T1} for us is that if $\mu$ is $n$-UR then for all $R \in \mathcal{D}(E)$ we have that
\begin{equation}\label{eq:acar}
\mathop{\sum_{Q \in \mathcal{D}(E)}}_{Q \subset R} \alpha(Q)^2 \mu(Q) \lesssim \mu(R).
\end{equation}
We also choose $c_Q$ and $L_Q$ which minimise $\alpha(Q)$. We always have that $c_Q \lesssim 1$, and that if $\alpha(Q)$ is small enough, then also $c_Q \gtrsim 1$ (see \cite{T1}).

We are now ready to formulate our theorem about square function estimates for UR measures. We discuss the context still a bit more after stating the theorem.
\begin{thm}\label{thm:main}
Let $n, d$ be integers and $0 < n < d$. Suppose $S$ is a kernel which satisfies $S \in K_{n+\beta}(\R^d)$ for some $\beta > 0$.
Let $\mu$ be an $n$-UR measure in $\R^{d}$ with $E = \textup{spt}\,\mu$. If the Carleson condition
\begin{equation}\label{eq:modcar}
\sup_{R \in \mathcal{D}(E)} \frac{1}{\ell(R)^n} \int_{\widehat R}\,  \sup_{L\colon d(x,L) \sim d(x,E)} |T_{S,L}1(x)|^2 d(x, E)^{2\beta-(d-n)}\, dx < \infty
\end{equation}
holds, then the square function estimate \eqref{eq:SF} holds. Here, for a given $x$, the supremum is taken over all the $n$-planes $L \subset \R^d$ for which $d(x,L) \sim d(x,E)$.
\end{thm}
The theorem is stated in all co-dimensions, but it is at least interesting in
the case of co-dimension 1 (meaning that $d = n + 1$).
Hofmann, Mitrea, Mitrea and Morris \cite{HMMM} proved that the square function estimate \eqref{eq:SF} holds,
if $\mu = \mathcal{H}^n_{\mid E}$ for a given $n$-UR set $E \subset \R^{n+1}$ and
$S(x,y) = (\partial_j K)(x-y)$ for some fixed $j \in \{1,\ldots, n+1\}$, where
$K \in C^2(\R^{n+1} \setminus \{0\})$, $K$ is odd and $K(\lambda x) = \lambda^{-n} K(x)$ for $\lambda > 0$, $x \in \R^{n+1} \setminus \{0\}$.
Such kernels $S$
are kernels of convolution form in $K_{n+1}(\R^{n+1})$ and in fact satisfy $T_{S,L}1(x) = 0$ for every $n$-plane $L$ and $x \not \in L$.
Much earlier, David and Semmes \cite{DS} had proved the case, where $K(x) = x_i/|x|^{n+1}$ for some $i \in \{1,\ldots, n+1\}$. 

The proof of Theorem \ref{thm:main} is completely different than the proof of the above referenced result in \cite{HMMM} (Corollary 5.7 of \cite{HMMM}).
The proof there follows the following steps:
\begin{enumerate}
\item Using a general local $Tb$ theorem one proves that having big pieces of square function estimates (BPSFE) is enough
for \eqref{eq:SF}. See Definition 4.1 of \cite{HMMM} for BPSFE.
\item One proves \eqref{eq:SF} in the case that $S(x,y) = (\partial_j K)(x-y)$ like above and $E$ is a Lipschitz graph. This uses, among other things, Fourier
analysis and borrows some techniques from the earlier papers \cite{Ch}, \cite{H1}, \cite{HL} and \cite{Jo1}.
\item One then considers a set $E$ which has big pieces of Lipschitz graphs (BPLG), or rather (BP)$^{k}$LG for some $k$ (big pieces of big pieces of...). Then \eqref{eq:SF}
follows (for $S$ like above) from the Lipschitz graph case using the theorem about the sufficiency of BPSFE.
\item Finally, one uses a deep geometric fact by Azzam and Schul \cite{AS} which says that a UR set $E$ has (BP)$^{2}$LG.
\end{enumerate}
We work directly with the given UR measure $\mu$ making no reductions. We use the technology of $\alpha$-numbers (inspired by \cite{CGLT}) not having to
resort to Fourier analysis. In particular, we don't have to restrict to convolution form kernels. We note that \cite{HMMM} does also include some
results about variable coefficient kernels for which they need their convolution form theorem combined with some
spherical harmonics extensions.

\begin{rem}
Notice that if $S \in K_{n+\beta}$, $\beta > 0$, has the cancellation property $T_{S,L}1 \equiv 0$ for every $n$-plane, then \eqref{eq:SF} holds
for every $n$-UR measure $\mu$. Our Carleson condition \eqref{eq:modcar} is a formal relaxation of this. The absolutely most naive relaxation would be to assume that
\begin{displaymath}
\int_{\R^{d} \setminus L} |T_{S, L}f(x)|^2 d(x, L)^{2\beta-(d-n)}\, dx \le C \int_L |f(y)|^2 \,d\mathcal{H}^n(y), \qquad f \in L^2(L),
\end{displaymath}
for every $n$-plane $L$ and some constant $C < \infty$ independent of $L$. But such a property does not imply much: the square function estimate
can then fail on a sphere even for a positive kernel $S$. For example, define
\begin{displaymath}
S(x,y) := \frac{1_{B_{1}}(x)}{H(x) + |x-y|^2}, \qquad x, y \in \R^2,\, x \ne y.
\end{displaymath}
Here $B_{1} := B(0,1)$ and $H\colon \R^2 \to [0,\infty]$, $H(x) := h(d(x,\partial B_{1}))$, where 
\begin{displaymath}
h(t) := t^2 \left(1 + \log^{+} \frac{1}{t} \right),
\end{displaymath}
and $\log^{+} t = \max\{\log t,0\}$ for $t > 0$. This was just a minor, perhaps obvious, side note and we omit all the details.
\end{rem}

\subsection*{Acknowledgements}
We thank  J. Azzam and T. Orponen for discussions related to square functions and  UR measures. T. Orponen constructed and verified the details of the omitted
example from the previous remark.

\section{Proof of the square function estimate}
In this section we give a proof of the square function estimate \eqref{eq:SF} under the UR hypothesis and \eqref{eq:modcar}.
The proof is short so we are quite generous with the details. 
\begin{proof}[Proof of Theorem \ref{thm:main}]
We will verify the $T1$ condition  \eqref{eq:t1}. To this end, fix $R \in \mathcal{D}(E)$. Recalling the definition of $\widehat R$ we need to prove that
\begin{displaymath}
\Car(R) := \mathop{\sum_{Q\in\mathcal{D}(E)}}_{Q \subset R} \ell(Q)^{2\beta-(d-n)} \int_{W_Q} |T_{S,\mu}1(x)|^2\, dx \lesssim \mu(R).
\end{displaymath}

For every $Q \in \mathcal{D}(E)$ and $x \in W_Q$ we want to prove that
\begin{align}\label{eq:main}
|T_{S,\mu}1(x)| & \lesssim \frac{1}{\ell(Q)^{\beta}} \mathop{\sum_{P \in \mathcal{D}(E) }}_{P \supset Q} \Big( \frac{\ell(Q)}{\ell(P)} \Big)^{\beta} \alpha(P)
+  \sup_{L\colon d(x,L) \sim d(x,E)} |T_{S,L}1(x)| \\
&=: U_1(Q) + U_2(x). \notag
\end{align}
Indeed, notice that
\begin{align*}
\mathop{\sum_{Q\in\mathcal{D}(E)}}_{Q \subset R} \ell(Q)^{2\beta-(d-n)} \int_{W_Q} U_1(Q)^2\, dx &\lesssim
\sum_{Q:\,Q \subset R} \mu(Q) \sum_{P:\, Q \subset P \subset R}  \Big( \frac{\ell(Q)}{\ell(P)} \Big)^{\beta} \alpha(P)^2 \\
&+ \sum_{Q:\,Q \subset R} \mu(Q) \sum_{P:\, R \subset P}  \Big( \frac{\ell(Q)}{\ell(P)} \Big)^{\beta} \alpha(P)^2 = I_1 + I_2.
\end{align*}
Using the Carleson property of the $\alpha$-numbers \eqref{eq:acar} we see that
\begin{displaymath}
I_1 = \sum_{P:\,P \subset R} \alpha(P)^2 \sum_{Q:\,Q \subset P} \Big( \frac{\ell(Q)}{\ell(P)} \Big)^{\beta} \mu(Q) \lesssim \sum_{P:\,P \subset R} \alpha(P)^2 \mu(P) \lesssim \mu(R).
\end{displaymath}
The term $I_2$ is completely elementary (we just estimate $\alpha(P) \lesssim 1$ and do not need the Carleson property of the $\alpha$-numbers):
\begin{displaymath}
I_2 \lesssim \sum_{Q:\,Q \subset R} \mu(Q) \Big( \frac{\ell(Q)}{\ell(R)} \Big)^{\beta} \sum_{P:\,R \subset P} \Big( \frac{\ell(R)}{\ell(P)} \Big)^{\beta} \lesssim \sum_{Q:\,Q \subset R} \mu(Q) \Big( \frac{\ell(Q)}{\ell(R)} \Big)^{\beta} \lesssim \mu(R).
\end{displaymath}

Next, notice that using \eqref{eq:modcar} we have that
\begin{align*}
\mathop{\sum_{Q\in\mathcal{D}(E)}}_{Q \subset R}& \ell(Q)^{2\beta-(d-n)} \int_{W_Q} U_2(x)^2\, dx \\
&\sim \int_{\widehat R}\,  \sup_{L\colon d(x,L) \sim d(x,E)} |T_{S,L}1(x)|^2 d(x, E)^{2\beta-(d-n)}\, dx  \lesssim \ell(R)^n \sim \mu(R).
\end{align*}
Combining the estimates, we have that \eqref{eq:main} implies that $\Car(R) \lesssim \mu(R)$. So it only remains to prove  \eqref{eq:main}.

Fix $Q \in \mathcal{D}(E)$ and $x \in W_Q$.
If $\alpha(Q) \ge c_0$ we have the trivial estimate
\begin{displaymath}
|T_{S,\mu}1(x)| \lesssim d(x, E)^{-\beta} \sim \ell(Q)^{-\beta} \lesssim \ell(Q)^{-\beta}\alpha(Q) \leq \frac{1}{\ell(Q)^{\beta}} \sum_{P:\,P \supseteq Q} \Big( \frac{\ell(Q)}{\ell(P)} \Big)^{\beta} \alpha(P).
\end{displaymath}
Therefore, we may assume that $\alpha(Q) < c_0$ for a small parameter $c_0$ to be chosen. We will use this to show that $B(x, c_1\ell(Q)) \cap L_Q = \emptyset$ (here $c_1>0$ is a small enough dimensional constant). To this end, let us first show that
\begin{equation}\label{eq:bilat}
\sup_{y\in L_Q\cap B(c_{Q},M\ell(Q))}\frac{d(y,E)}{M\ell(Q)} \le C c_0^{1/(n+1)} = \epsilon_0.
\end{equation}
Suppose $y\in B(c_{Q},M\ell(Q))\cap L_{Q}$ and set $\tau=\tau_y= d(y,E)/M \ell(Q)$. Notice here that $\tau \in [0,1]$ and $B(y,\tau M\ell(Q))\subset \R^{d} \setminus E$. Let $\phi$ satisfy
$1_{B(y,\tau M\ell(Q)/2)}\leq \phi\leq 1_{B(y,\tau M\ell(Q))}$ and Lip$(\phi) \sim (\tau \ell(Q))^{-1}$. Note that spt$\,\phi \subset B(c_{Q},2M\ell(Q)) = B_Q$ and spt$\,\phi \subset \R^{d} \setminus E$.
Therefore, we have that
\begin{align*}
\tau^{-1}c_0 \ell(Q)^n \gtrsim \textup{Lip}(\phi) \alpha(Q) \ell(Q)^{n+1} \ge c_Q \int \phi\, d\mathcal{H}^n_{\mid L_Q} \gtrsim_M \tau^n \ell(Q)^n.
\end{align*}
Here we used that $c_Q \gtrsim 1$ since $\alpha(Q)$ is small.
This establishes \eqref{eq:bilat}.

Suppose then that $B(x, c_1\ell(Q)) \cap L_Q \ne \emptyset$. Then there exists $W \in \mathcal{W}$ so that $Q(W) = Q$ and there exists $y \in 2W \cap L_Q \subset B(c_Q, M\ell(Q)) \cap L_Q$ (the constant $c_1$ is so small that  $B (x, c_1 \ell(Q)) \subset 2W$ if $x \in W$).
But, in view of \eqref{eq:bilat}, this means that
\begin{displaymath}
\ell(W) \sim d(y,E) \leq \epsilon_0 M \ell(Q) \sim \epsilon_0\ell(W),
\end{displaymath}
which is a contradiction for a small enough $\epsilon_0$ i.e. for a small enough $c_0$. We thus conclude that $B(x, c_1\ell(Q)) \cap L_Q = \emptyset$.

Recalling that $c_Q \lesssim 1$ we estimate
\begin{displaymath}
|T_{S,\mu}1(x)| \lesssim \Big| \int S(x,y)\,d(\mu - c_Q \mathcal{H}^n_{\mid L_Q})(y)\Big | + \Big| \int S(x,y)\,d\mathcal{H}^n_{\mid L_Q}(y)\Big | = P_1 + P_2.
\end{displaymath}
We first deal with $P_1$. To this end, notice that
\begin{displaymath}
P_1 \le \sum_{k \ge 1} \Big| \int \gamma_k(y) S(x,y)\,d(\mu - c_Q \mathcal{H}^n_{\mid L_Q})(y)\Big |,
\end{displaymath}
where $\sum_{k \ge 0} \gamma_k = 1$, $\gamma_k$ is smooth, supported on those $y$ for which $|x-y| \sim 2^k \ell(Q)$ and satisfies $\|\nabla \gamma_k \|_{\infty} \lesssim (2^k\ell(Q))^{-1}$.
The key thing is that the corresponding function $\gamma_0$ is not needed, since it is supported on $B(x, c_1\ell(Q))$ and this does not intersect the support of the measure i.e. $E \cup L_Q$.
We further estimate
\begin{align*}
P_1 \le \sum_{k \ge 1} \Big|& \int \gamma_k(y) S(x,y)\,d(\mu - c_{Q^{(k+s_0)}} \mathcal{H}^n_{\mid L_{Q^{(k+s_0)}}})(y)\Big | \\
&+ \sum_{k \ge 1} \Big| \int \gamma_k(y) S(x,y)\,d(c_{Q^{(k+s_0)}} \mathcal{H}^n_{\mid L_{Q^{(k+s_0)}}} - c_Q \mathcal{H}^n_{\mid L_Q})(y)\Big | = J_1 + J_2
\end{align*}
for some $s_0 \sim 1$ such that spt $\gamma_k \subset B_{Q^{(k+s_0)}}$.

The function $y \mapsto \gamma_k(y) S(x,y)$ is Lipschitz with
\begin{displaymath}
\textup{Lip}( \gamma_k(\cdot) S(x,\cdot)) \lesssim (2^k\ell(Q))^{-n-\beta-1}.
\end{displaymath}
This follows easily by using the size, Lipschitz and support properties of the involved functions.
Therefore, we have that
\begin{displaymath}
J_1 \lesssim \sum_{k \ge 1} \alpha(Q^{(k+s_0)}) \ell(Q^{(k+s_0)})^{n+1}(2^k\ell(Q))^{-n-\beta-1} \lesssim \frac{1}{\ell(Q)^{\beta}} \sum_{P:\,P \supset Q} \Big( \frac{\ell(Q)}{\ell(P)} \Big)^{\beta} \alpha(P).
\end{displaymath}

Let us then estimate $J_2$. Let $f_k(y) := (2^k\ell(Q))^{n+1+\beta}\gamma_k(y)S(x,y)$ so that
\begin{displaymath}
J_2 \le \sum_{k \ge 1} (2^k\ell(Q))^{-n-1-\beta} \sum_{j=1}^{k + s_0} \Big| \int f_k(y)\,d(c_{Q^{(j)}} \mathcal{H}^n_{\mid L_{Q^{(j)}}} - c_{Q^{(j-1)}} \mathcal{H}^n_{\mid L_{Q^{(j-1)}}})(y)\Big |.
\end{displaymath}
For a fixed $k$ and $j$ we estimate
\begin{align*}
\Big| \int f_k(y)&\,d(c_{Q^{(j)}} \mathcal{H}^n_{\mid L_{Q^{(j)}}} - c_{Q^{(j-1)}} \mathcal{H}^n_{\mid L_{Q^{(j-1)}}})(y)\Big |\\ 
&\le c_{Q^{(j)}} \Big| \int f_k(y)\,d(\mathcal{H}^n_{\mid L_{Q^{(j)}}} - \mathcal{H}^n_{\mid L_{Q^{(j-1)}}})(y)\Big| \\
&+ |c_{Q^{(j)}} -c_{Q^{(j-1)}}| \Big| \int f_k(y)\, d\mathcal{H}^n_{\mid L_{Q^{(j-1)}}}(y)\Big | = K_1 + K_2.
\end{align*}

To estimate $K_1$ we use that $c_{Q^{(j)}} \lesssim 1$ and that
\begin{displaymath}
d_{H}(L_{Q^{(j-1)}} \cap B_{Q^{(k+s_0)}} , L_{Q^{(j)}}\cap B_{Q^{(k+s_0)}}) \lesssim \alpha(Q^{(j)}) \ell(Q^{(k+s_0)}).
\end{displaymath}
Here $d_H$ stands for the Hausdorff distance and we have used Lemma 3.4 of \cite{T1}. Using this one sees that
\begin{displaymath}
K_1 \lesssim \alpha(Q^{(j)}) (2^k\ell(Q))^{n+1}.
\end{displaymath}
Lemma 3.4 of \cite{T1} also gives that $|c_{Q^{(j)}} -c_{Q^{(j-1)}}| \lesssim \alpha(Q^{(j)})$. After this it is clear that also
\begin{displaymath}
K_2 \lesssim \alpha(Q^{(j)}) (2^k\ell(Q))^{n+1}.
\end{displaymath}

We conclude that
\begin{align*}
J_2 &\lesssim \sum_{k \ge 1} (2^k\ell(Q))^{-\beta} \sum_{j=1}^{k + s_0} \alpha(Q^{(j)}) \\
&\lesssim \sum_{R:\, R \supset Q} \ell(R)^{-\beta} \sum_{P:\, Q \subset P \subset R} \alpha(P) \\
&= \sum_{P:\, P \supset Q}\alpha(P) \ell(P)^{-\beta} \sum_{R:\, R \supset P} \Big(\frac{\ell(P)}{\ell(R)}\Big)^{\beta} \\ &\lesssim \sum_{P:\, P \supset Q}\alpha(P) \ell(P)^{-\beta} = 
 \frac{1}{\ell(Q)^{\beta}} \sum_{P:\,P \supset Q} \Big( \frac{\ell(Q)}{\ell(P)} \Big)^{\beta} \alpha(P).
\end{align*}
Combining the estimates for $J_1$ and $J_2$ we have that
\begin{displaymath}
P_1 \lesssim \frac{1}{\ell(Q)^{\beta}} \sum_{P:\,P \supset Q} \Big( \frac{\ell(Q)}{\ell(P)} \Big)^{\beta} \alpha(P).
\end{displaymath}

For $P_2$ notice that we have that $d(x, L_Q) \sim \ell(Q) \sim d(x,E)$. Therefore, we have that
\begin{displaymath}
P_2 \lesssim \sup_{L\colon d(x,L) \sim d(x,E)} |T_{S,L}1(x)|.
\end{displaymath}
Combining the estimates for $P_1$ and $P_2$ we have that \eqref{eq:main} holds. Therefore, we are done.

\end{proof}


\begin{bibdiv}
\begin{biblist}

\bib{AS}{article}{
   author = {Azzam, Jonas},
   author = {Schul, Raanan},
   title = {Hard Sard: quantitative implicit function and extension theorems for Lipschitz maps},
   journal={Geom. Funct. Anal.},
   volume={22},
   date={2012},
   number={5},
   pages={1062--1123},
}

\bib{Ch}{book}{
   author={Christ, Michael},
   title={Lectures on singular integral operators},
   series={CBMS Regional Conference Series in Mathematics},
   volume={77},
   publisher={Published for the Conference Board of the Mathematical
   Sciences, Washington, DC; by the American Mathematical Society,
   Providence, RI},
   date={1990},
   pages={x+132},
}

\bib{CGLT}{article} {
   author={Chousionis, Vasilis},
   author={Garnett, John},
   author={Le, Triet},
   author={Tolsa, xavier},
   title={Square functions and uniform rectifiability},
   journal={preprint},
   eprint={http://arxiv.org/abs/1401.3382}
} 

\bib{DS}{book}{
   author={David, Guy},
   author={Semmes, Stephen},
   title={Analysis of and on uniformly rectifiable sets},
   series={Mathematical Surveys and Monographs},
   volume={38},
   publisher={American Mathematical Society, Providence, RI},
   date={1993},
   pages={xii+356},
}

\bib{H1}{article}{
   author={Hofmann, Steve},
   title={Parabolic singular integrals of Calder\'on-type, rough operators,
   and caloric layer potentials},
   journal={Duke Math. J.},
   volume={90},
   date={1997},
   number={2},
   pages={209--259},
}

\bib{HL}{article}{
   author={Hofmann, Steve},
   author={Lewis, John L.},
   title={Square functions of Calder\'on type and applications},
   journal={Rev. Mat. Iberoamericana},
   volume={17},
   date={2001},
   number={1},
   pages={1--20},
}

\bib{HMMM}{article}{
   author={Hofmann, Steve},
   author={Mitrea, Dorina},
   author ={Mitrea, Marius},
   author={Morris, Andrew J.},
   title = {$L^p$-square function estimates on spaces of homogeneous type and on uniformly rectifiable sets},
   journal={preprint},
   eprint={http://arxiv.org/abs/1301.4943},
,}

\bib{Jo1}{article}{
   author={Jones, Peter W.},
   title={Square functions, Cauchy integrals, analytic capacity, and
   harmonic measure},
   conference={
      title={Harmonic analysis and partial differential equations (El
      Escorial, 1987)},
   },
   book={
      series={Lecture Notes in Math.},
      volume={1384},
      publisher={Springer, Berlin},
   },
   date={1989},
   pages={24--68},
}

\bib{T1}{article} {
   author={Tolsa, Xavier}
   title={Uniform rectifiability, Calder\'on--Zygmund operators with odd kernel, and quasiorthogonality},
   journal={Proc. London Math. Soc.}
    volume={98},
   date={2009},
   number={2},
   pages={393--426},
}   


\end{biblist}
\end{bibdiv}
\end{document}